\documentclass[12pt]{amsart}
\usepackage[utf8]{inputenc}
\usepackage{pstricks}
\usepackage{amsfonts}
\usepackage{amssymb}
\usepackage{amsthm}
\usepackage{latexsym,amsmath}
\usepackage{graphicx}
\usepackage{stmaryrd}

\newtheorem{theorem}{Theorem}[section]
\theoremstyle{definition}
\newtheorem{Definition}[theorem]{Definition}

\newenvironment{theorem*}[1]{\medskip
                            \noindent
                            {\bf Theorem #1. }\ %
                            \begingroup \sl}
                            {\endgroup\medskip}

\title{Brouwer fixed point theorem as a corollary of Lawvere}

\author[$\mathrm{M^{\lowercase{c}}Callum}$ ]{\textbf{Rupert} $\mathbf{M^{\lowercase{c}}Callum}$ }

\begin{document}

\maketitle

\begin{abstract} It is investigated in what sense the Brouwer fixed point theorem may be viewed as a corollary of the Lawvere fixed point theorem. A suitable generalisation of the Lawvere fixed point theorem is found and a means is identified by which the Brouwer fixed point theorem can be shown to be a corollary, once an appropriate continuous surjective mapping $A' \rightarrow X^{A''}$ has been constructed for each space $X$ in a certain class of ``nice" spaces for each one of which the exponential topology on $X^{A''}$ exists, and here $A'$ and $A''$ have the same carrier set and the topology on $A'$ is finer than on $A''$. It is shown that there is a certain natural way of attempting to derive Brouwer as a corollary of Lawvere which is not possible, namely that is there is no space $A$ for which the exponential topology on $[0,1]^{A}$ exists and there is a continuous surjection $A \rightarrow [0,1]^{A}$. We then examine the range of contexts in which phenomena like those described in the first result occur, from a broadly model-theoretic perspective, with a view towards applications for the original motivation for the problem as a problem in decision theory for AI systems, suggested by the Machine Intelligence Research Institute. \end{abstract}

\noindent The Brouwer fixed point theorem, whose first published proof was in \cite{Hadamard1910}, states that if $D$ is a closed ball in a finite-dimensional Euclidean space then every continuous mapping $D \rightarrow D$ has a fixed point. The Lawvere fixed point theorem first appeared in \cite{Lawvere1969}. Let us recall the statement of the Lawvere fixed point theorem and its various applications. The theorem states that, in a Cartesian closed category in which there is a point-surjective morphism $A \rightarrow X^{A}$, every endomorphism of $X$ has a fixed point. This is in some sense ``the essence of" all diagonal arguments, having as corollaries Cantor's theorem, the diagonal lemma which is used in the proof of G\"odel's theorem, unsolvability of the halting problem, and existence of a program in any computer language which outputs its own source code.

\bigskip

\noindent It is of interest to know whether the Brouwer fixed point theorem in topology can be recovered as a special case of the Lawvere fixed point theorem; the purpose of this note is to examine various senses in which this is and is not the case.

\bigskip

\noindent One might first naturally ask whether there exists some space $A$, and some class of spaces including every closed ball in a finite-dimensional Euclidean space, with the property that, if $X$ is a space in the class, then the exponential topology on $X^{A}$ exists and there is a continuous surjection $A \rightarrow X^{A}$. We shall later show that it is provable in ZFC that this is not the case even if we just restrict to $X=[0,1]$.

\bigskip

\noindent Suppose two spaces $A'$ and $A''$ could be identified, with the same carrier set and the topology on $A'$ finer than on $A''$, that is to say there is a point-bijective morphism $h: A' \rightarrow A''$. Suppose that some class of spaces could be found, including every closed ball in a finite-dimensional Euclidean space, with the following properties. If $X$ is a space in the class, then the exponential topology on $X^{A''}$ exists and there is a continuous surjection $g:A' \rightarrow X^{A''}$, such that, if $U \subseteq X$ is open, and eval:$X^{A''} \times A'' \rightarrow X$ is the evaluation map, and $\pi_{1}$ is the projection from $X^{A''} \times A''$ onto the first factor, then if we let $V=g^{-1}(\pi_{1}(\mathrm{eval}^{-1}(U)))$, then both $V$ is open in $A'$ and $h(V)$ is open in $A''$. If such a surjection can be found for every space in our class, then the proof of the Lawvere fixed point theorem applies to show that every continuous endomorphism of every space $X$ in the class has a fixed point and so the Brouwer fixed point theorem is recovered as a special case. This shall be our strategy in the next two sections, where we will indeed recover Brouwer as a corollary of an appropriate generalisation of Lawvere in this way. We must begin by defining the appropriate class of spaces.

\section{Defining the appropriate class of topological spaces}

\noindent Let us begin by stating the appropriate generalisation of Lawvere; we must translate the previously given discussion into category-theoretic terms.

\bigskip

\begin{theorem} Suppose that $C$ is a category with a terminal object and closed under products, and $S$ is a class of objects of $C$. Suppose that $A'$ and $A''$ are objects of $C$ with a point-bijective morphism $h:A' \rightarrow A''$. Suppose that, for every $X \in S$, the exponential object $X^{A''}$ exists (exponential object relative to some full Cartesian-closed subcategory of $C$, held fixed throughout). Suppose that for each such $X$ there is a point-surjective morphism $g:A' \rightarrow X^{A''}$, with the property that every morphism $A' \rightarrow X$ which factors through $g \times h:A' \rightarrow X^{A''} \times A''$ also factors through $h:A' \rightarrow A''$. Then every object $X$ is such that every morphism from $X$ to itself has a fixed point. \end{theorem}

\begin{proof} Same proof as proof of standard Lawvere fixed point theorem. \end{proof}

In order to recover the Brouwer fixed point theorem as a special case of this generalisation of Lawvere, we must define the class of objects $S$, and the objects $A'$ and $A''$ for which we intend to apply it.

\bigskip

\noindent Let $\omega_{1}$ denote the first uncountable ordinal, and consider the complete infinite binary tree of height $\omega_{1}$, viewed as a generalised Cantor space. So this means that the carrier set is the set of all strings of length $\omega_{1}$ of symbols from the set $\{0,1\}$, with the topology for which the basic open sets are sets consisting of all the strings starting with a fixed string of countable length as an initial fragment. This space will be denoted by $A'$. The space with the same carrier set, with the product topology obtained when the space is viewed as a product of $\aleph_1$ many discrete two-element spaces in the natural way, will be denoted by $A''$. Now consider the class $S$ of all spaces $X$ satisfying the following conditions.

\begin{Definition} A topological space $X$ is said to be nice and a member of $S$ if the following hold.

\bigskip

\noindent First, the space $X$ is compact and contractible and is the image of the generalised Cantor space described above under a continuous surjection. Secondly, the space $X$ is the disjoint union of an open dense subset $U$ and the complement $V$, and both $U$ and $V$ admit a ``homogenous" metric, where what we mean by this is as follows. Firstly, with regard to $U$, there exists some $\epsilon>0$, with the property that, for all $\delta$ such that $0<\delta<\epsilon$, every pair of distinct open balls of radius $\delta$ centred at a point in $U$ such that the closure of the balls does not intersect $V$, has the property that the balls in the pair are isometric. Then, with regard to $V$, we require that sufficiently small open balls in $X$, centred at points of $V$, of the same radius, are isometric.

\end{Definition}

\noindent Clearly the class $S$ of nice spaces so defined includes all closed balls in finite-dimensional Euclidean spaces. Since $A''$ and all spaces in $S$ are all $k$-spaces the existence of all the needed exponential topologies is clear (as the category of $k$-spaces is in fact a full Cartesian-closed subcategory of Top).

\bigskip

\noindent In an earlier version of this argument, we thought that the condition of contractibility would end up playing an essential role in the proof. It now appears that the condition involving the existence of a certain type of metric is really the key consideration, and most likely the condition of contractibility follows from this anyway, and is therefore redundant. In any case the inclusion of the condition of contractibility can be dispensed with in what follows; no essential use of that condition will be made.

\bigskip

\noindent In addition, we must show that there is a continuous surjection $A' \rightarrow X^{A''}$ for every nice space $X$, satisfying the requirements given in the statement of the generalised Lawvere fixed point theorem given above. It will then follow by the generalised Lawvere fixed point theorem that every nice space $X$ is such that every continuous function $X \rightarrow X$ has a fixed point, including all closed balls in finite-dimensional Euclidean spaces as a special case, thereby completing the first part of our argument. We can now state our main theorem with regard to the application of the generalisation of Lawvere given above.

\section{Brouwer as a direct corollary of a generalisation of Lawvere}

\begin{theorem} Let $A'$ and $A''$ be as in the previous section, suppose that $X$ is a nice topological space. Then there is a continuous surjection $A' \rightarrow X^{A''}$, where $X^{A''}$ has the exponential topology, satisfying the requirements given in the statement of the generalised Lawvere fixed point theorem stated in the previous section. As a corollary of this, the Brouwer fixed point theorem can be recovered as a corollary of the generalised Lawvere fixed point theorem. \end{theorem}

\begin{proof}

\noindent We must construct the continuous surjection $g: A' \rightarrow X^{A''}$ on the stated hypotheses. Suppose that $x \in A'$; note that $x$ can be thought of as a bit-string of length $\omega_{1}$, and we must describe how to choose $g(x)$. Let us begin by making some observations about how one might code for an element of $X^{A''}$.

\bigskip

\noindent It is evident that given that $X$ is a nice space, the space $X$ has the cardinality of the continuum. It is also possible to construct an $\omega$-sequence $T:=\{C_{n}:n \in \omega\}$ of coverings of $X$ by finitely many open balls, each covering $C_{n} \in T$ being such that it can be partitioned into two collections of open balls with each collection having the property that all of the balls in it are pairwise isometric, and also such that the mesh of the covering $C_{n}$ tends towards zero as $n$ goes to infinity. Now suppose that we have a mapping which sends each point of a countable subset $C \subseteq A''$ to a centre of some open ball appearing in some $C_{n} \in T$, with every centre of every such open ball appearing in the range of the mapping. If an extension of this mapping to a continuous element of $X^{A''}$ exists, there will be some countable ordinal $\alpha$, with the property that, given any bit-string of length $\alpha$ and considering the set of all elements of $A''$ which have this bit-string as an initial fragment, the extension in question will be constant on this set.

\bigskip

\noindent So we see from this that a coding scheme can be constructed whereby every element of $X^{A''}$ can be coded for by a countable bit-string (not necessarily unique given the initial choice of element of $X^{A''}$). First we describe the general form that the data to be used in this coding will take, where again, data of this form can always be constructed for any given element of $X^{A''}$, but not uniquely. Consider a map $\theta$ whose domain $B$ is a countable collection $B$ of countable bit-strings, closed under taking initial fragments, and such that every element of $2^{\omega_{1}}$ has a fragment which is the union of a branch of $B$, and whose co-domain is $X$. We can further require that in the case of bit-strings of zero or successor length, the value of the map at these bit-strings is always a centre of an open ball from some $C_{n} \in T$. Next, use the axiom of choice to construct a function $\rho$ defined on the set of all countable limits of limit ordinals, whose value at each such ordinal $\alpha$ is an $\omega$-sequence of limit ordinals cofinal in $\alpha$. We can now also make the further requirement that given any branch in $B$, the trace of the mapping $\theta$ along this branch, a well-ordered countable sequence of elements of $X$ is ``generalised Cauchy of degree $n$" for some positive integer $n$ which is the same for all branches. What this means is that, for each fragment of the branch of limit length, we obtain an $\omega$-sequence of ordinals cofinal in the length of the fragment, either in the obvious way, if it is a limit ordinal which is not a limit of limit ordinals, or via the previously constructed function $\rho$ otherwise, and then we require that the trace of the mapping $\theta$ along the nodes of the fragment of the branch of $B$ indexed by this $\omega$-sequence of ordinals, is Cauchy relative to the metric on $X$ which we have been holding fixed throughout, and with the speed of convergence having a uniform lower bound determined by $n$, say for example that if $m, m' \geq k$ and $x_m, x_m' \in X$ are the elements of $X$ corresponding to the $m$-th and $n$-th ordinals in the sequence then $d(x_m, x_m')<\frac{1}{2^{n+k}}$ where $d$ is the metric on $X$. (We are requiring $n$ to be positive for the moment, but shall later need to generalise to situations where we have the same criterion with $n$ allowed to be zero or negative.) Let $E$ be the set of all tuples $(B,\theta,n)$ where $B, \theta$ and $n$ are as described above. We now wish to describe a coding scheme whereby every such tuple can be coded for by a countable well-ordered bit-string.

\bigskip

\noindent The integer $n$ can clearly be coded for by a finite bit-string. Each point of $X$ which is a centre of one of the open balls appearing in one of the coverings $C_{n}$ can be coded for by a finite bit-string according to some fixed coding scheme, and in the case where we are dealing with an arbitrary point of $X$, our requirements entail that there is no need to include a code for this point of $X$, since it will be possible to infer it from data occurring earlier in the bit-string that is coding for our tuple $(B,\theta,n)$. The original sequence $T=\{C_{n}:n\in\omega\}$ of coverings can be chosen in such a way that it is indeed always possibile to be able to satisfy the ``generalised Cauchy criterion" at each limit stage of each branch of $B$; this is a consequence of the homogeneity assumptions we made on the metric.

\bigskip

\noindent Thus, our coding scheme will be such that under the coding scheme a countable well-ordered bit-string codes for a positive integer $n$ together with a map $\theta$ from a countable collection $B$ of countable well-ordered bit-strings into points of $X$ with the constraints described before, where the only points of $X$ that actually need coding can be coded for by finite bit-strings.

\bigskip

\noindent The exact details of the coding scheme are not all that important but we shall introduce a couple of extra requirements on it later on which will be easily seen to be possible to fulfil. Firstly, we will want to require that the set $D$ of countable well-ordered bit-strings which can serve as codes in the coding scheme is such that no two countable bit-strings in $D$ are such that one is a fragment of the other, every bit-string of length $\omega_1$ has a bit-string in $D$ as a fragment, and every bit-string in $D$ is infinite. It is clear that this requirement can be fulfilled. Then our energies will be occupied with showing that every element of $E$ does code for an element of $X^{A''}$, and at that point it will be clear that a continuous surjection $A'\rightarrow X^{A''}$ can be constructed from the coding scheme in a natural way (namely, given an element of $A'$ which is a bit-string of length $\omega_1$, find the unique element $b \in D$ which is a fragment of it, and map the element of $A'$ to the element of $X^{A''}$ corresponding to the element of $E$ which is coded for by $b$). We shall want the coding scheme to be constructed not only in such a way as to satisfy the previously given requirements, but also in such a way that the surjection constructed from the coding scheme in the obvious way meets the requirements stated in the statement of the generalised Lawvere fixed point theorem, relative to the space $X$. That is to say, if one takes an open subset of $X$, then takes the pre-image in $X^{A''} \times A''$ under the evaluation map, the projection of that onto the first factor, and then the pre-image of that under the surjection $A' \rightarrow X^{A''}$, which one views as both a subset of $A'$ and $A''$, this set must be open relative to both topologies. In the case of $A'$, the basic open sets for the topology are countable intersections of sets whose defining condition is given by specifying the value of just one bit at the $\alpha$-th point in the sequence for some countable ordinal $\alpha$, whereas in the case of $A''$, the basic open sets are finite intersections of such sets. We must require that the set we obtain be open in the latter topology (the former one being finer). Once we have established that every element of $E$ does indeed code for an element of $X^{A''}$, the possibility of constructing the coding scheme in such a way that this requirement is fulfilled easily follows from the details already given.

\bigskip

\noindent Suppose again that $(B,\theta,n) \in E$. The values of the mapping $\theta$ at the ends of the branches of $B$ determine a mapping $A'' \rightarrow X$ in an obvious way, and clearly for every element of $X^{A''}$ there is indeed at least one element of $E$ that represents it under this coding scheme. We shall say more presently of how we might be able to ensure that the induced mapping $A'' \rightarrow X$ so defined is continuous if we are starting with an arbitrary element of $E$.

\bigskip

\noindent Given any element of $E$, there exists a countable ordinal $\alpha$ such that the mapping $\theta$ coded for by the said element of $E$ extends uniquely to a mapping $2^{\leq\alpha}\rightarrow X$, satisfying the aforementioned generalisation of the Cauchy criterion along each fragment of a branch of limit length. Uniquely, that is, if we introduce the extra constraint that for elements of $2^{\leq\alpha}$ not in $B$ the mapping has the same value as for the union of the branch of $B$ consisting of fragments which are in $B$. What we now need to see is that, given the hypothesis of compactness of $X$, and the ``generalised Cauchy criterion" on the mapping $\theta$, every element of $E$ does indeed determine an element of $X^{A''}$, that is the function from $A''$ into $X$ which is thereby defined is indeed continuous.

\bigskip

\noindent Using our chosen coding scheme, every element of $E$ clearly gives rise to a mapping $2^{\leq\alpha}\rightarrow X$, for an appropriately chosen countable ordinal $\alpha$, as previously noted, and proving continuity of the restriction of this to $2^{\alpha}$ (with the obvious product topology on $2^{\alpha}$) is sufficient. Suppose that $\beta\leq\alpha$ and consider the induced mapping $2^{\beta} \rightarrow X$. We must show by transfinite induction that this is always continuous, relative to the generalised Cantor space topology on $2^{\beta}$, given the hypothesis of compactness and contractibility on $X$ and the generalised Cauchy criterion on $\theta$. So, suppose that $\beta \leq \alpha$ and the desired conclusion has been established for all ordinals less than $\beta$. Clearly we only need to consider the case where $\beta$ is a limit ordinal, and we have a mapping $2^{\leq\beta}\rightarrow X$ available, which we shall use freely in what follows, satisfying the generalised Cauchy criterion which we have introduced.

\bigskip

\noindent The induction hypothesis states that the map $2^{\gamma} \rightarrow X$ is continuous for all $\gamma<\beta$, and since $X$ is compact, this holds with a modulus of uniform continuity depending only on $\gamma$. Here the modulus of uniform continuity can be a constant of Lipschitz continuity relative to any metrics on $2^{\gamma}$ and $X$ that we wish. Naturally the metric on $X$ that we wish to employ is the one that we have been holding fixed all along, and the choice of metric on $2^{\gamma}$ does not matter as long as it is compatible with the obvious product topology on $2^{\gamma}$, and chosen so that there is a choice of Lipschitz constant which works for all $\gamma<\beta$. From this consequence of the induction hypothesis together with our ``generalised Cauchy criterion", it can be concluded that the map $2^{\beta} \rightarrow X$ is continuous and so now our transfinite induction goes through, yielding the final conclusion that the mapping $2^{\alpha} \rightarrow X$ is continuous, and therefore that the induced mapping $A'' \rightarrow X$ is. So this means that our coding scheme is indeed such that every element of $E$ does indeed give rise to an element of $X^{A''}$, and conversely every element of $X^{A''}$ arises from (at least one) element of $E$. Now we need to construct a coding scheme whereby elements of $E$ can be coded for by countable well-ordered bit-strings from an appropriate collection $D$, satisfying all the previously mentioned constraints. With it now established that every element of $E$ does indeed code for an element of $X^{A''}$ there is no difficulty in seeing that all the constraints on the coding scheme can indeed be satisfied.

\bigskip

\noindent The existence of a surjection $A' \rightarrow X^{A''}$ with all of the desired properties is now clear. It follows by the generalised Lawvere fixed point theorem that every nice space $X$ is such that every continuous endomorphism has a fixed point. Thus the Brouwer fixed point theorem can be recovered as a corollary of this generalisation of Lawvere.

\end{proof}

\section{Brouwer as a direct corollary of ordinary Lawvere is not possible}

\noindent It is natural to ask whether there exists a space $A$ with the property that for all spaces $X$ in an appropriate class, the exponential topology on $X^{A}$ exists and there is a continuous surjection $A \rightarrow X^{A}$. Establishing this for a class that included every closed ball in a finite-dimensional Euclidean space would give us a way of recovering the Brouwer fixed point theorem as a direct corollary of the ordinary Lawvere fixed point theorem. But in fact we shall now see how to prove in ZFC that even just in the case $X=[0,1]$ there is no such space.

\bigskip

\begin{theorem} It is provable in ZFC that there does not exist any topological space $A$ for which the exponential topology on $[0,1]^{A}$ is defined and such that there is a continuous surjection $g:A \rightarrow [0,1]^{A}$. \end{theorem}

\begin{proof} Suppose that the space $A$ and the surjection $g$ exist in $V$ and define $B:=(A,g)$. Consider the inner model $M:=L(\mathbb{R})[B]$. A Skolem hull argument shows that we can find a $B'=(A',g')$ with the same properties relative to $L(\mathbb{R})$, such that $B' \in L_{\alpha}(\mathbb{R})$ for some countable limit $\alpha$. Let $\rho$ be the function constructed earlier, restricted to $\alpha$. There exists a real number $r$ such that each continuous endomorphism $f$ of $[0,1]$ in $L[B',\rho]$ is in $L_{\alpha}[r]$. We follow arguments given by Stephen Simpson in \cite{Simpson1999} for equivalence of Brouwer's fixed point theorem with $WKL_{0}$ over $RCA_{0}$. Choose a structure $S$ for the second-order language of arithmetic which is an $\omega$-model and which includes a code for every element of $L_{\alpha}[r]$ in the domain of the number variables and items coding for the structure $L_{\alpha}[r]$ relative to this coding scheme appear in the domain of the set variables, and such that the structure is a model for $RCA_{0}$ but not $WKL_{0}$. This is possible by the model-theoretic results about $RCA_{0}$ and $WKL_{0}$ presented in Chapter VIII of \cite{Simpson1999}; namely, one selects the least ``Turing ideal" in $\mathcal{P}(\omega)$ which satisfies the constraints just given, and this will model $RCA_{0}$ but not $WKL_{0}$. Then, following the proof of Theorem IV.7.7 in \cite{Simpson1999}, construct a continuous endomorphism of $[0,1]$ in $S$ which has no fixed point in $L_{\alpha}[r]$. But this is a contradiction because of the constructive nature of the proof of the Lawvere fixed point theorem. If $B'$ occurs in $L_{\alpha}[r]$ and therefore has its transitive closure (as an $\in$-structure) fully coded for in $S$, and $B'$ has the stated properties relative to $L(\mathbb{R})$, then there should be no difficulty in the proof of an existence of a fixed point in $S$, but in fact existence of the fixed point fails in $S$.

\end{proof}

\noindent This proof basically says that, since the Lawvere fixed point theorem is constructive, the non-constructive nature of any procedure for constructing the fixed point in any proof of the Brouwer fixed point theorem would have to be in some sense ``entirely coded for" by the surjection $g$. This can work where $g:A' \rightarrow X^{A''}$ for two spaces $A'$ and $A''$ with the same carrier set and different topologies, but not when $A'$ and $A''$ are the same space.

\section{The original motivation for the problem}

\noindent The Machine Intelligence Research Institute originally became interested in this problem motivated by concerns in decision theory. Suppose that we have an agent with two actions $A$ and $B$ available to them, and they must choose a probability distribution for which action to perform which depends in a continuous way on which observation they make from the space $X$ of all possible observations. Then the set of policies available is $[0,1]^{X}$ (let us assume that the exponential topology does indeed exist). It may be that the agent is observing another agent in the environment so the space $X$ of all possible observations is equal to the space of all possible agents. If each agent has a well-defined policy, then there is a mapping $X \rightarrow [0,1]^{X}$ which we might reasonably require to be continuous. It is of interest from a decision-theoretic point of view to know whether for some spaces $X$ it is possible for such a continuous mapping to be surjective.

\bigskip

\noindent We have obtained the answer ``No, such spaces do not exist, but if you are happy with there being a finer topology on the $X$ that occurs on the left than on the $X$ that occurs in the exponent on the right, then yes the existence of such a continuous surjection is possible". We were dealing with spaces of cardinality $2^{\aleph_1}$ which are perhaps too big to be plausible candidates for practical applications, so it may be of interest to explore whether other examples can be given of more modest cardinality. Let us attempt to survey the range of examples that can occur, looking at it from the perspective of model theory.

\bigskip

\noindent Suppose that $S \subseteq \mathcal{P}(\omega)$ is a Turing ideal. Countable well-ordered bit-strings can be coded for by elements of $\mathcal{P}(\omega)$, and a certain set $A$ of countable well-ordered bit-strings can be coded for by elements of $S$. Note that the supremum of the lengths of bit-strings in $A$ may be strictly less than the $\omega_1$ of ``the real world". Consider the $\omega$-model for the second-order language of arithmetic determined by $S$. The set $A$ can clearly be coded as a ``definable sub-class" of the range of the second-order variables on this structure. The ``generalised Cantor space topology" and ``compact product topology" on $A$ are likewise both clearly definable by means of ``iterated predicative comprehension" on this structure. However, if $RCA_{0}$ is our base theory then we will not go beyond being conservative over $RCA_{0}$ by introducing iterated predicative comprehension for higher types. Clearly, some additional axiom is needed along the lines of ``an uncountable subtree of the complete binary tree of height $\omega_1$, every level of which is countable, has an uncountable branch". (Of course, an investigation from the point of view of Reverse Mathematics naturally suggests itself here.)

\bigskip

\noindent If we introduce an axiom like that in the third-order part of the language, then the desired surjection $g:A' \rightarrow X^{A''}$ will be found in the domain of the third-order variables. But of course there is no guarantee that every real number recursively constructible from $g$ occurs in the original Turing ideal $S$. But if $S$ is chosen to have the appropriate closure properties, then this will be the case and then an appropriate ``space of agents" can be constructed from $S$ which has the desired properties relative to a particular model (and also relative to any strictly larger Turing ideal with the same closure properties, including all of $\mathcal{P}(\omega)$). Relative to such a model you will be able to find a ``space of agents" (possibly countable ``in the real world") and a continuous surjection (relative to the model, and also many strictly larger models as just indicated) from the space of agents onto the space of those policies with a certain upper bound on their complexity. The ``space of agents" is also uniformly definable across the whole class of models for which the construction works, and the bound on complexity of the bit-strings occurring in the space of agents can be strictly larger than the bound on complexity of the ordinals which index the entries in the bit-string. So you can indeed enlarge the space of agents so that it has cardinality $2^{\aleph_0}$ and arbitrary elements of $\mathbb{R}$ are definable from the bit-strings that occur in it, but the gap between the upper bound on complexity of elements of $X^{A''}$ and the complexity of $g$ will still hold.

\bigskip

\noindent The hierarchy of complexity can be refined further, going to considerations of computational complexity rather than just descriptive complexity. We can go to a weakened version of $RCA_{0}$ which only allows for elementary recursive comprehension and then substitute ``Turing reducibility" with ``Turing reducibility via elementary recursive functions". Or for ``elementary recursive" one could substitute ``polynomial-time computable". (In these cases, clearly the height of the binary tree $A$ would become considerably shorter. I previously thought that this kind of considerations might lead to situations where the height of the binary tree is only a recursive ordinal but Alex Mennen has sketched a proof that this is not the case and one only obtains shortening of the height to $\omega_{1}^{CK}$, thereby possibly reducing the interest of these considerations.)

\bigskip

\noindent So, for example, one could have a function from the space of agents to the space of polynomial-time computable policies, but without any possibility of organising things so that every agent is coded for by a real number that is itself polynomial-time computable, or so that the function from the space of agents to the space of polynomial-time computable policies is polynomial-time computable. This starts to sound a bit more like the kind of application that would be of interest from the point of view of the study of artificial intelligence. The Machine Intelligence Research Institute may wish to explore further ramifications of the results presented here along the lines just suggested.

\end{document}